\documentclass[12pt]{article}
\usepackage{float}
\usepackage{array}
\usepackage{amssymb,latexsym, amsmath}
\usepackage{amsmath}
\usepackage[a4paper, total={6in, 8in}]{geometry}
\usepackage{tabularx}
\usepackage{hyperref}
\usepackage{url}
\usepackage{bm}
\usepackage{graphicx}
\usepackage{hyperref}
\usepackage{subcaption}
\usepackage{euscript}
\usepackage{hyperref}
\usepackage{multicol}

\usepackage{xcolor}
\usepackage{tikz}
\usepackage{tikz-qtree}
 \usepackage{epstopdf}
 \usepackage{float}
\usepackage{graphicx}
 \usepackage[applemac]{inputenc}
\usepackage{theorem}
\usepackage{epsfig}
\usepackage{authblk}
\usepackage{amsmath}

\usepackage{tikz} 

\usepackage{tikz-qtree}
\usetikzlibrary{trees}

\usetikzlibrary{arrows.meta, calc, positioning}

\DeclareSymbolFont{AMSb}{U}{msb}{m}{n}
\DeclareSymbolFontAlphabet{\Bbb}{AMSb}

\newtheorem{theorem}{Theorem}[section]
\newtheorem{lemma}{Lemma}[section]
\newtheorem{proposition}{Proposition}[section]
\newtheorem{corollary}{Corollary}[section]

{\theorembodyfont{\rmfamily}\newtheorem{remq}{Remark}[section]}

{\theorembodyfont{\rmfamily}\newtheorem{definition}{Definition}[section]}

\newcounter{example}[section]
\newenvironment{example}[1][]
   {\refstepcounter{example}\par\noindent
    \textbf{Example \thesection.\arabic{example}} \textit{#1}\ }

\newenvironment{remark}
 {\begin{remq}}
 {\par\end{remq}}

\newenvironment{proof}{\paragraph{Proof:}}{\hfill$\square$}

\newcommand{\VV}{\mathcal{V}}
\newcommand{\Ss}{\mathcal{S}}
\newcommand{\OO}{\mathcal{O}}

\font\ly=lasy10 scaled 1100
\chardef\lg='050
\chardef\rg='051
\def\ofd{\leavevmode\hbox{\ly\lg\kern-0.2em\lg}}
\def\ffg{\leavevmode\hbox{\ly\rg\kern-0.2em\rg}}

\begin{document}

\author[1]{A. Srinivasan\thanks{ORCID: 0000-0002-7964-433X}}
\author[2]{L. A. Calvo\thanks{ORCID: 0000-0002-8818-5542}}

\affil[1,2]{Universidad Pontificia de Comillas, ICADE, Alberto Aguilera, 23, 28015 Madrid, Spain}



\title{Minimal triples for a generalized Markoff equation}
\maketitle

\providecommand{\keywords}[1]
{
  \small	
  \textbf{\textit{Keywords--}} #1
}

\begin{abstract}
For a positive integer $m>1$, if the generalized Markoff equation $a^2+b^2+c^2=3abc+m$ has a solution triple, then it has infinitely many solutions.  We show that all  positive solution triples are generated by a finite set of  triples that we call minimal triples. We exhibit a  correspondence between the set of minimal triples with first or second element equal to $a$,  and the set of  fundamental solutions of $m-a^2$ by the form $x^2-3axy+y^2$. This gives us a formula for the number of minimal triples in terms of fundamental solutions, and thus 
 a way to calculate minimal triples using composition and reduction  of binary quadratic forms, for which there  are efficient algorithms. Additionally, using the above correspondence we also give a criterion for the existence of minimal triples of the form $(1, b, c)$, and  present a formula for the number of such minimal triples. 
\end{abstract}

\keywords{Markoff triples, fundamental solutions, generalized Markoff equation.}

\section{Introduction}
A  Markoff triple is a solution $(a,b,c)$ of positive integers of the equation 
\begin{equation}\label{Moriginal}
a^2+b^2+c^2=3abc.
\end{equation}
These triples made their first appearance in the work of Markoff \cite{M2} on the minima of quadratic forms. Since then this remarkable equation has been studied in a variety of ways.  In this work, we consider the generalized  equation
    \begin{equation}\label{Mm}
    a^2+b^2+c^2=3abc+m,
  \end{equation}
 where $m>1$ is a positive integer.
 Mordell \cite{Mor} analysed the general equation $x^2+y^2+z^2=axyz+b$ by considering solution triples of three types and outlining ways to find them. However, he does not give a method to find all solution triples.
 Equation (\ref{Mm}) has been studied in works such as \cite{TWX}, \cite{GS}  and \cite{LM}. These authors, among other things, were interested in values of $m$ for which there are no solutions (Hasse failures). Our objective is to study a special set of positive solution triples of  (\ref{Mm}), that we call minimal triples. We define a minimal triple $(a, b, c)$ as a positive ordered solution triple for which $3ab-c\leq0$. Our focus will be on minimal triples of the kind $(1, b, c)$, with the first component equal to $1$. In particular, we are interested in the case when there is exactly one minimal triple. 
  Minimal triples generate all positive solution triples of (\ref{Mm}) that we call  $m$-Markoff triples.  Moreover, they satisfy the condition that $a^2+b^2\le m$, which allows us to find them explicitly.
  
  Markoff exhibited a tree containing all solution triples of equation (\ref{Moriginal}). It is not so well known that this is also the case for equation (\ref{Mm}). 
 In contrast to the Markoff equation, in this case, there could be more than one tree of solution triples.
Each $m-$Markoff triple is found on a tree of solutions. 
 Furthermore, each minimal triple generates a distinct tree of solution triples, enabling us to count the number of trees as given in the following theorem.

\begin{theorem}\label{th1.1} Every positive solution triple of (\ref{Mm})  is contained in a unique tree. Furthermore, the number of solution trees is equal to the number of minimal triples.
\end{theorem}
\color{black}
Minimal triples are connected in a very natural way to fundamental solutions of representations by binary quadratic forms as follows.
Let 
 $$F(x,y)=x^2-3axy+y^2$$ be a binary quadratic form of discriminant $d=9a^2-4$.
 We may re-write equation (\ref{Mm}) as $b^2-3abc+c^2=m-a^2$, that is, 
 $F(b,c)=m-a^2$. Therefore every solution triple $(a, b, c)$ of (\ref{Mm}) gives rise to a representation of $m-a^2$ by the form $F(x,y).$ For a given $a$, all representations $(x,y)$ such that $F(x,y)=m-a^2$ may be put into equivalence classes, where in each class there is a unique representation called the fundamental solution (see Theorem \ref{MRS}).  Now, if $(a,b,c)$ is minimal, then it is associated with a unique fundamental solution (see (\ref{FS})). This sets up a correspondence between minimal triples $(a, b,c)$ (with $a$ fixed) and fundamental solutions of $F(x,y)=m-a^2$, which allows us to find and count these triples efficiently (as there are fast reduction algorithms to find fundamental solutions). To understand better the correspondence mentioned above, we define the set 
 \begin{equation}\label{Ta}
     T_a=\{(a, b,c): (a, b, c) {\text { or }} (b, a, c) {\text{ is a minimal triple}} \}.
 \end{equation}
 Note that the cardinality of $T_a$ is the number of minimal triples that have $a$ as the first or second component.
In Theorem 1.2  below, we present a formula that connects the number of fundamental solutions to the number of minimal triples and the number of improper minimal triples (triples with equal first and second components).

 \begin{theorem} \label{th1.2}
     Let $m>1$ and $0<a<\sqrt{m}$.
      Suppose that 
       $S_a$ is the set of all  fundamental solutions   of $F(x,y)=m-a^2$. Then $S_a$ and $T_a$ have the same cardinality. Moreover 
       $$ \#\{S_a: a<\sqrt{m}\}=2 \#
      \{ \text{minimal triples }\}  -
      \#\{{\text{improper minimal triples}}.
       \} $$
      
  \end{theorem}
  \begin{remark}
      The proof of the above theorem is achieved by defining a bijective mapping between the sets $T_a$ and $S_a$ (see (\ref{FS})). Thus, to find the set of minimal triples, we find all fundamental solutions
      of $m-a^2$ by $F(x,y)$ (for each $a<\sqrt{m}$). This involves composition and reduction of binary quadratic forms, for which we have efficient algorithms. Moreover, it is conjectured that for a given $m$, the number of minimal triples is $\ll_{\epsilon} m^{\epsilon}$ (\cite[Conjecture 10.1]{GS}). This means that we look for solutions only for a small number of $a's$. For example, for 
      $m=480492$, there are only $4$ minimal triples, that we are able to find in a few seconds using Theorem 1.2 (a brute force search would take much longer).
  \end{remark}
  \color{black}
   If the integer $m-a^2$ is represented by some form of discriminant $d$, then we can determine the total number of fundamental solutions by all forms that represent $m-a^2$ (this depends on the number of prime divisors of $m-a^2$). Only those fundamental solutions that correspond to representations by $F(x,y)$ will give rise to minimal triples $(a, b,c)$.  In the case when $a=1$ (or $d=5$), all fundamental solutions correspond to $F(x,y)$ as there is only one form in the class group of $\mathbb{Q}(\sqrt{5})$.
Also, the set $T_1$ contains all minimal triples  which contain $1$ as the first or second component. It follows that the number of minimal triples $(1, x, y)$ is equal to the cardinality of  $T_1$, and hence of $S_1$ by Theorem 1.2 above. 
    As a result, in the following theorem, we are able to give a formula for the number of minimal triples of the kind $(1, b, c)$.  
  \begin{theorem}\label{th1.3}
     Let 
        $m>1$. Let $w(N)$ denote the number of distinct primes in $N$ and let $(\frac{N}{5})$ represent the Legendre symbol.  Then the following hold.
        \begin{enumerate} 
        \item 
        There exists an $m$-Markoff triple $(1,b,c)$  if and only if $m-1=S^2C$, where $C$ is square-free such that 
         if $p$ is prime and $p|C$ then 
            $(\frac{p}{5})\ne -1$.
           \item Suppose that there exists an $m$-Markoff triple $(1,b,c)$.  Let $m-1=5^{2\alpha}A^2B^2C$, where  $\alpha\ge 0$ and $C$ satisfies the conditions given in 1. Furthermore, assume that if $p$ is prime with 
           $p|B$, then $(\frac{p}{5})=1$ and if
           $p|A$, then $(\frac{p}{5})= -1$. 
            Then the number of minimal triples $(1,b,c)$ is equal to 
        $\sum_{d|B} 2^{w\left(\frac{B^2C}{d^2}\right)+l-1}$, where 
         $l=(\frac{C}{5})$. 
         \end{enumerate}      
\end{theorem}
One may pose several interesting questions about minimal triples, such as whether there are infinitely many $m$'s with exactly one minimal triple. This is the subject of the  Section 5 where we make several conjectures on minimal triples of the kind $(1, b, c)$.

  The outline of the paper is as follows. In Section 2 we present results on minimal triples and associated trees. Section 3  contains the theory of binary quadratic forms and fundamental solutions. In Section 4 we give the proofs of the main theorems. In Section 5 we pose some conjectures and questions, and in the last section we present  computations that support our conjectures.

%
%
%
%

\section{Minimal  triples and trees}
Henceforth $m$ will denote a positive integer greater than $1$.

An $m$-{\it Markoff triple} is a solution $(a,b,c)$ 
 of positive integers satisfying
     \begin{equation*}
    x^2+y^2+z^2=3xyz+m.
  \end{equation*}
It is {\it proper} if $a,b,c$ are distinct and {\it improper} if it is not proper. A triple $(a,b,c)$ is {\it ordered} if $a\leq b \leq c$.

As in the case of Markoff triples, each solution triple $(a,b,c)$ of (\ref{Mm}) has three neighbouring triples, obtained by applying the Vieta involutions given below.

\begin{definition}\label{vieta}
Let $(a,b,c)$ be a solution triple for (\ref{Mm}). The {\it Vieta involutions } $\VV_1,\VV_2,\VV_3$ of $(a,b,c)$ are also solution triples:
\begin{align*}
\VV_1(a,b,c)=(3bc-a,b,c),\\
\VV_2(a,b,c)=(a,3ac-b, c),\\
\VV_3(a,b,c)=(a,b,3ab-c).
\end{align*}
\end{definition}

In the lemma below we give a property that is true for the usual ordered Markoff triples. 
The proof in \cite[Lemma 2.1]{LS} of this property essentially works for our generalized equation. However, as there are some minor differences in the proof, we present it here, especially since while the property is not new, it applies to a new equation.
\begin{lemma}\label{fact}
    If $(a, b, c)$ is an ordered $m$-Markoff triple, then $3ab< b+c$.
\end{lemma}
\begin{proof}
    Let us first consider the case when $a=b=c$. We have in this case 
    $3a^2=3a^3+m$, which gives
    $a=1$, which is not possible,
    as $(1,1,1)$ is not an $m$-Markoff triple when $m>0$.

     Next, let us assume that $a<b=c$. Then from equation (\ref{Mm}) we have $a^2+2c^2=3ac^2+m$, which means $3ac^2< 3c^2$ (as $a<c$). It follows that $a=0$, which is not possible. 
     
     Therefore we have $b<c$ and from (\ref{Mm}) we have $3abc<3c^2$, and hence
     $a^2\le ab<c$.  It follows that 
     $a^2+b^2<c+b^2$ and hence 
     $\frac{a^2+b^2}{c}<1+\frac{b^2}{c}<1+b$ (as $b<c$) and we have 
     \begin{equation}\label{one}
     \frac{a^2+b^2}{c}< b+1.
     \end{equation}
     Now, from  (\ref{Mm}) and   (\ref{one}) above, we have 
     \begin{equation}
         3ab<\frac{a^2+b^2}{c}+c< b+c+1
     \end{equation}
      and hence 
     \begin{equation}\label{puff}
     3ab\le b+c.
     \end{equation}
      To conclude the proof we will show that $3ab=b+c$ is not possible. Let us suppose on the contrary, that for an ordered triple  $(a, b, c)$ we have  $3ab=b+c$. It follows that $\VV_3(a, b, c)=(a, b, b)$ is an ordered triple and on applying (\ref{puff}) to this triple, we obtain $3ab\le 2b$, which is not possible. Hence the inequality of the lemma holds.
\end{proof}

\begin{definition}\label{minimaltriple}
An  $m$-Markoff triple $(a,b,c)$ is {\it minimal} if $a\le b\le c$ and 
$$ 3ab-c\leq 0.$$
\end{definition}

Recall that improper triples are those for which all three components are not distinct. These triples when ordered are minimal as seen in the following result.
\begin{proposition}\label{improper}
    Let $(a, b, c)$ be an ordered improper $m$-Markoff triple. Then $a=b$ and $(a, a, c)$  is minimal.
\end{proposition}
\begin{proof}
    Let $(a, b, c)$ be an ordered $m$-Markoff triple that is improper, that is, its three components are not distinct. If either  $a=b=c$ or $b=c$, then the inequality   $3ab<b+c$ from 
     Lemma \ref{fact} yields $3a<2$, which is not possible. Hence we may assume that $a=b<c$, so that by   (\ref{Mm}) we have
     $2a^2+c^2=3a^2c+m$ or 
     \begin{equation}\label{improper2}
     2a^2-m=c(3a^2-c).
     \end{equation}
     If $2a^2-m>0,$ then $c| 2a^2-m$ and 
     hence $c\le 2a^2-m< 2a^2$. From Lemma \ref{fact}, we have $3a^2<a+c$. Combining the last two equations,  we have 
     $3a^2< a+2a^2$, a contradiction.
Hence $2a^2-m\le 0$, 
     and we have 
     $3a^2-c<0$ from (\ref{improper2}). It follows 
     that $(a,a,c)$ is minimal by definition.
\end{proof}

\vspace{0.5cm}
\noindent In the following lemma we collect some bounds for minimal triples.

\begin{lemma}\label{abless} Let $(a,b,c)$ be an ordered  $m$-Markoff triple. Then the following hold.
\begin{enumerate}
\item $(a,b,c)$ is minimal if and only if $a^2+b^2\leq m.$
\item If $(a,b,c)$ is minimal then $1\leq a\leq \sqrt{\frac{m}{2}}$. 
\item If $(a,b,c)$ is minimal then $c>\sqrt{m}$ and if $c\ne 3ab$, then $c<m$.

\item If $(a,b,c)$ is minimal then $3ab\leq c\leq 3ab+\sqrt{m-a^2-b^2}.$

\end{enumerate}
\end{lemma}

\begin{proof}

\begin{enumerate}
\item Recall that by definition $(a,b,c)$ is minimal if and only if $c-3ab\geq 0.$ Hence the first statement is a consequence of re-writing (\ref{Mm}) as 
    \begin{equation}\label{mc}
    a^2+b^2+c(c-3ab)=m.
    \end{equation}

\item If $(a,b,c)$ is minimal, by part 1 proved above, we have $a^2+b^2\leq m$. Since $(a,b,c)$ is ordered, we have $1\leq a\leq b$, so $2a^2\leq m$, giving $a\leq \sqrt{\frac{m}{2}}$.

\item For any two natural numbers $n_1,n_2$  it is true that $$n_1^2+n_2^2<9n_1^2n_2^2.$$
    Taking $n_1=a, n_2=b$ and assuming that $3ab\leq c$ (as $(a,b,c)$ is minimal) we have
    $$a^2+b^2<3abc.$$
    Thus, since $a^2+b^2-3abc<0$, we have
    $$c>\sqrt{c^2+(a^2+b^2-3abc)}=\sqrt{m}.$$
    
    Finally, if $(a,b,c)$ is   minimal, with
    $c\ne 3ab$, then $c-3ab>0$ and from (\ref{mc}) we obtain $c\leq c(c-3ab)<m$.

    \item If $(a,b,c)$ is minimal, then $3ab\leq c$ and so $0\le c-3ab<c$. Therefore 
    $c-3ab\le \sqrt{c(c-3ab)}=\sqrt{m-a^2-b^2}.$
\end{enumerate}
\end{proof}

\color{black}

\begin{remark}
    Note that the above lemma allows us to compute all minimal triples, for a given $m>0$. Indeed, by statements 1 and 2, we  obtain the bounds $1\leq a\leq \sqrt{\frac{m}{2}}$ and $a\le b\le \sqrt{m-a^2}.$ On the other hand, $c$ is also bounded according to 4. Now we can iterate through all the triples within the bounds, checking if they satisfy equation (\ref{Mm}).
    However it should be noted that for large values of $m$ this brute force algorithm is inefficient. Theorem \ref{th1.2} gives us another way to compute minimal triples using fundamental solutions, which involves reduction algorithms of binary quadratic forms which are typically more efficient.    
   
\end{remark}

\begin{remark} There are natural numbers $m$ for which there are no m-Markoff triples (and hence no minimal triples). Looking at equation  (\ref{Mm}) modulo $4$, it is easy to see that if $m\equiv 3\pmod {4}$, then (\ref{Mm}) has no solutions. Indeed, there are infinitely many such $m$ as shown in \cite{TWX} and \cite{GS}.
On the other hand, in the case when $m$ is a sum of two non-zero squares, the set of minimal triples is non-empty. This is because
    each representation of $m$ as a sum of two non-zero squares, say $m=a^2+b^2$, corresponds to a minimal triple $(a, b, 3ab)$. 
\end{remark}

The following lemma is crucial to prove Theorem \ref{th1.2}, as it is used to set up a correspondence between a set of minimal triples and a set of fundamental solutions.
\begin{lemma}\label{either}
    Let $(a, b, c)$ be
    an $m$-Markoff triple such that either $(a, b, c)$ or $(b, a, c)$ is minimal.  Then  one of  $b$ or $c-3ab$ is less than or equal to $\sqrt{\frac{m-a^2}{3a+2}}.$ 
\end{lemma}
\begin{proof}
    By definition of minimality, we have
    $3ab\le c$. If $3ab-c=0$, then we are done and hence we assume  that 
    $c-3ab>0.$

      Suppose that both 
       $b$ and $c-3ab$ are greater than $\sqrt{\frac{m-a^2}{3a+2}}$. 
    Re-writing (\ref{Mm}) as
    $b^2+c(c-3ab)=m-a^2$, we have
    $$ \frac{m-a^2}{3a+2}+c \sqrt{\frac{m-a^2}{3a+2}}<m-a^2$$
    or
    $$c \sqrt{\frac{m-a^2}{3a+2}}<\frac{3a+1}{3a+2} (m-a^2),$$
    which gives 
\begin{equation}\label{cless}c<\frac{(3a+1)\sqrt{m-a^2}}{\sqrt{3a+2}}.
    \end{equation}
    Now $$c-3ab<c-3a\sqrt{\frac{m-a^2}{3a+2}}<
    \sqrt{\frac{m-a^2}{3a+2}}$$
    (using  inequality (\ref{cless}) above), which contradicts our assumption and hence the result follows.
    \end{proof}
    
\vspace{0.5cm}

We now look at solution trees, for which we first define the root of a tree.

\begin{definition}\label{root}
Let $(a,b,c)$ be a minimal triple. Then the {\it root} of the associated tree is given by
$$
\begin{cases}
    (a, b, c) & {\text{ if } (a, b, c) \text{ is proper }}\\
    (a,c,3ac-b) & {\text{ if } (a, b, c) \text{ is improper. }}
\end{cases}
$$
\end{definition}

\vspace{0.5cm}
\begin{example}
The triple $(1,2,6)$ is a root because it is a proper minimal $5$-Markoff triple. The triple $(1,5,14)$ is a root because, although it is not minimal, it arises from the $12$-Markoff triple $(1,1,5)$, which is an improper minimal triple.
\end{example} 

\vspace{0.5cm}

The tree of solutions with root $(a,b,c)$ is constructed as follows: if $(x,y,z)$ is an $m$-Markoff triple, the nodes coming out are $(x,z,3xz-y)$ and $(y,z,3zy-x)$.

As a particular example, in Fig.\ref{5-markovtree}, we display the beginning of the $5$-Markoff triple with root $(1,2,6)$. 

\color{black}
\begin{figure}[H]
\centering

\begin{tikzpicture}[grow'=right,level distance=1.5in, sibling distance=.15in]
\tikzset{edge from parent/.style= 
            {thick, draw, edge from parent fork right},      every tree node/.style={
			draw,
			rounded corners,
			anchor = west,
			text width=25mm,
			align=center}}
   \Tree 
    [. (1,2,6) 
        [.(1,6,16)
            [.(1,16,42) 
                    [.(1,42,110)$\cdots$  ]
                    [.(16,42,2015)$\cdots$ ]
            ]
            [.(6,16,287)
                    [.(16,287,13770)$\cdots$ ]
                    [. (6,287,5150)$\cdots$ ]
            ]
        ]
        [.(2,6,35)
            [.(6,35,628) 
                    [.(6,628,11269)$\cdots$  ]
                    [.(35,628,65934)$\cdots$ ]
            ]
            [.(2,35,204) 
                    [.(35,204,21418)$\cdots$ ]
                    [.(2,204,1189)$\cdots$ ]
            ]
        ] 
    ]
\end{tikzpicture}
    \caption{Beginning of the $5$-Markoff tree with root $(1,2,6)$.}
    \label{5-markovtree}
\end{figure}
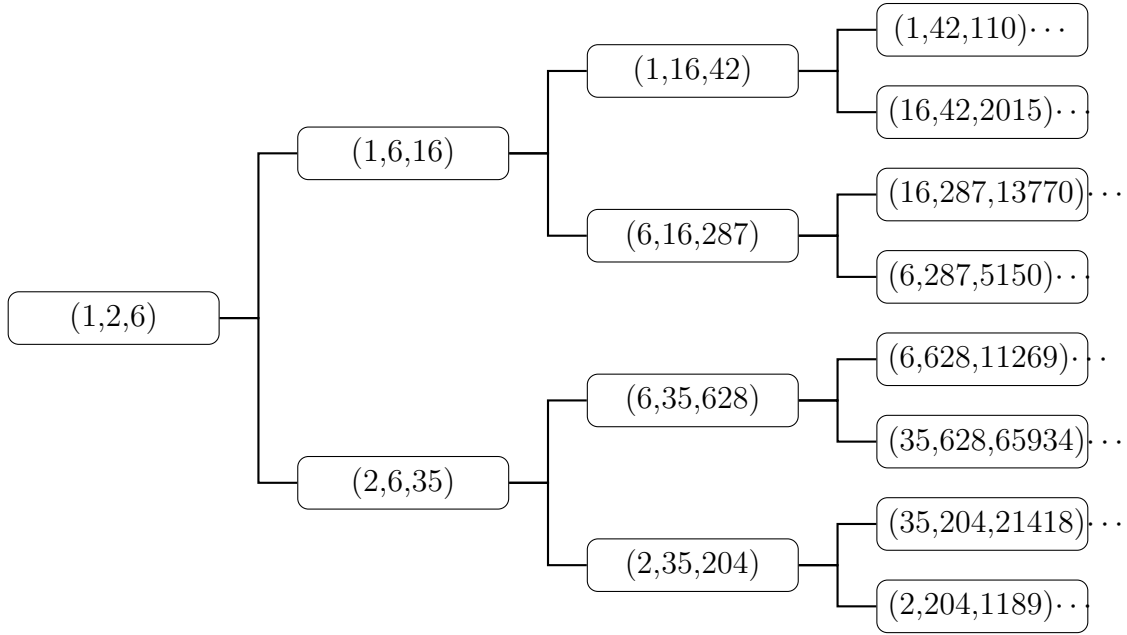

\vspace{0.5cm}
\begin{example} In the case $m=5$, there is only one root and hence only one tree of solutions. For $m=4,$ there are no solutions and for $m=12$, there are two roots $(1,5,14)$ and $(1,2,7)$, and thus two different solution trees.
\end{example}

\vspace{0.5cm}

\section{Binary quadratic forms and fundamental solutions}
 In this section, we present the basic theory of
binary quadratic forms. An excellent  reference
for this topic is \cite{Ri}, where in particular, the reader may 
consult Chapter 6, Sections 4 to 7 for the material presented here. 

\subsection{Binary quadratic forms}

A {\sl primitive binary quadratic form} $f=(a, b, c)$ of discriminant $d$ is
a function $f(x, y)=ax^2+bxy+cy^2$, where $a, b,c $ are integers
with $b^2-4 a c=d$ and $\gcd(a, b, c)=1$.  Note that the discriminant $d$ is always $0$ or $1$ mod $4$.
 All forms considered here
are primitive binary quadratic forms and henceforth we shall
refer to them simply as forms. 

Two forms $f$ and $f'$ are said to be {\it equivalent}, written as
$f\sim f'$,  if for some
$A=\begin{pmatrix} \alpha &\beta \\ \gamma & \delta \end{pmatrix}
\in SL_2(\mathbb Z)$ (called a transformation matrix),  we have
$f'(x,y)=f(\alpha x+\beta y, \gamma x+\delta y)
=(a',b',c')$, where 
$a', b', c'$ are given by
\begin{equation}
a'=f(\alpha, \gamma),\hskip2mm
b'=2(a\alpha\beta+c\gamma\delta)+b(\alpha\delta+\beta\gamma),\hskip2mm
c'=f(\beta, \delta).
\end{equation}
 It is easy to see
that $\sim$ is an equivalence relation on the set of forms of
discriminant $d$. The equivalence classes form an abelian group
called the  {\it class group} with group law given by composition of
forms.

The {\it identity form} is defined as the form $(1,0,\frac{-d}{4})$
or $(1, 1, \frac{1-d}{4})$, depending on whether $d$ is even or odd
respectively. 
 The {\it inverse} of
$f=(a, b, c)$ denoted by $f^{-1}$, is given by $(a,-b,c).$

A form $f$ is said to represent an integer $m$ if there exist 
integers $x$ and $y$ such that $f(x,y)=m$. If $\gcd(x, y)=1$, we 
call the representation a primitive one.
Observe that equivalent forms primitively represent the same set
of integers, as do a form and its inverse. Observe that the identity form represents the integer $1$. Moreover, any form that represents $1$ is equivalent to the identity form.

\color{black}

The following lemma tells us when an integer is represented by a form of a given discriminant.
\begin{lemma}\label{4n}\cite{Ri}[ Solution of Problem 1] Let
    $d\equiv 0$ or $1\mod 4$.  
        Then there exists a primitive representation of an integer $N$ by a form
        of discriminant $d$ if and only if
        $d\equiv x^2\pmod{4N}$ for some integer $x$.

\end{lemma}
\subsection{Fundamental solutions}
It is well known that all representations of an integer $N$ by a given binary quadratic form may be put into equivalence classes. 
In this section we consider representations of  $m-a^2$ by the form $x^2-3axy+y^2$, namely,
\begin{equation}\label{fa}
F(x,y)=x^2-3axy+y^2=m-a^2,
\end{equation}
where the form in question is of discriminant $d=9a^2-4$, for $0<a<\sqrt{m}$.
In 
 each equivalence class there is a unique fundamental solution $(u, v)$ with least non-negative value of $v$. The following result from \cite[Theorem 4.1]{MRS} (modified to fit our case) gives us the fundamental solutions of (\ref{fa}), where we have used the fact that the fundamental solution of the Pell equation $x^2-dy^2=4$ is $(x,y)=(3a, 1)$. 
\begin{theorem}\label{MRS}\cite[Theorem 4.1]{MRS} Let $m>1$ and 
$a<\sqrt{m}$ be positive integers.
Let $V=\sqrt{(m-a^2)/(3a+2)}$ and  $U=\sqrt{(m-a^2)(3a+2)}$. 
Then a solution
$(u, v)$ with $v\ge 0$ of (\ref{fa}) is a fundamental solution 
if and only if one of the following holds:
   \begin{enumerate}
   \item $0 < v < V$.
   \item $v=0$ and $u=\sqrt{m-a^2}$.
   \item $v=V$ and $u=(U+3aV)/2$.
   \end{enumerate}
\end{theorem}

\begin{corollary}\label{Corollary3.1} Let $m>1$ and $a<\sqrt{m}$. 
Let $V=\sqrt{\frac{m-a^2}{3a+2}}$ and  $U=\sqrt{(m-a^2)(3a+2)}$. Let 
 $(a, b, c)$ be an  $m-$Markoff triple such that either $(a, b, c)$ or $(b,a,c)$ is minimal. Suppose that  either $b=V$
 or $c-3ab=V$. Then
$c=3ab+b$ and 
$(c, b)$ is a fundamental solution for (\ref{fa}) with $N=m-a^2$.
\end{corollary}
\begin{proof} Let us first assume that $b=V$.  Then 
$$b^2=\frac{m-a^2}{3a+2}=\frac{b^2+c^2-3abc}{3a+2}$$
giving  
$$b^2(3a+1)=c^2-3abc$$
or
$$3ab^2+b^2=c^2-3abc.$$ 
It follows that 
$$c^2-b^2=3ab(b+c)$$
and hence
\begin{equation}\label{abc}
  c=3ab+b.  
\end{equation}
 As $b=V$, from Theorem \ref{MRS}, part 3), to conclude the proof we will show that 
$c=\frac{U+3aV}{2}$. We have 
$U=b(3a+2)$ (follows from $V=b$) and hence 
$\frac{U+3aV}{2}=\frac{b(3a+2)+3ab}{2}=3ab+b=c$ (from (\ref{abc})).

The case when $c-3ab=V$ follows on applying the above proof to the minimal triple $(a, c-3ab, 3a(c-3ab)+b)$.
\end{proof}
\color{black}

The following lemma is well known and the result is classical. However as a clear reference seems to be lacking, we provide proof.
\begin{lemma}\label{wn} Let $N>1$ be a positive integer and let $w(N)$ denote the number of distinct prime divisors of $N$. 
Suppose that $N=AB^2$, where $A$ is square-free. Then there exists a primitive representation
$F(x,y)=N$ if and only if  $5\nmid B$ and 
    $(\frac{p}{5})\ne -1$
    for every $p|N$. Furthermore, the number of fundamental solutions is equal to 
    $2^{w(N)}$ if $N\not\equiv 0\pmod{5}$ and 
    equal to $2^{w(N)-1}$ if $5|N$.
\end{lemma}
\begin{proof}  
We start with the observation that  $F(x,y)$ is the only form in the class group here (as the class number of $\mathbb{Q}(\sqrt{5})$ is $1$. Therefore all fundamental solutions of $N$ correspond to $F(x,y)$.
Next, by Lemma \ref{4n} we have that  $F(x,y)=N$ is a primitive representation if and only if  
  $5\equiv l^2\pmod{4N}$ for some integer $l$. It follows that 
  there exists a primitive representation  $F(x,y)=N$ if and only if 
  $(\frac{5}{p})\ne -1$
    for every $p|N$.  Moreover, $25\nmid N$ as 
    $5\equiv l^2\pmod {25}$ has no solutions, and thus $5\nmid B$. It is well known that every fundamental solution $F(x,y)=N$ corresponds to a solution 
    $d\equiv X^2\pmod{4N}$, where
    $1\le X\le 2N$ and vice-versa (see for example, \cite{Ri}[ Solutions of Problem 2, 3, Problems 4, 5, pages 120-121] and 
 \cite{MRS}[(1.1)-(1.4)]). Therefore the number of fundamental solutions is equal to the number of solutions of this congruence. It follows from elementary number theory (\cite{HW}[Theorem 122]) that the number of solutions of this congruence is as stated in the lemma. 
\end{proof}

\section{Proofs of the main theorems}

\noindent{\bf Proof of Theorem \ref{th1.1}}
  We start with the observation that if $(a, b, c)$ is an ordered $m$-Markoff triple, then $\VV_3(a, b, c)=(a, b, 3ab-c)$ has maximal element $b$ (by Lemma \ref{fact}). Continuing in this way (as long as all components of the triple are positive), ordering the triples each time and applying the Vieta involution $\VV_3$, we will arrive at a triple $(a_0, b_0,c_0)$ such that $c_0\le 0$. It follows that
 the triple $\VV_3(a_0, b_0, c_0)$ is minimal and hence the given triple is on the tree with root associated to this minimal triple (see Definition \ref{root}).
To complete our proof, we observe that an $m$-Markoff triple cannot belong to two different trees by the reasoning above, as each tree has a unique root.  
\hfill$\square$

\vspace{0.7cm}

\noindent{\bf Proof of Theorem \ref{th1.2}} 

\noindent Let $0<a<\sqrt{m}$ be a fixed integer. Note that $T_a$(defined in (\ref{Ta})) contains all the minimal triples that contain $a$ as the first or second component, reordering the triple in the latter case, so as to have the first component as $a$. We will show there is a one-to-one correspondence between $T_a$ and $S_a$, the set of all fundamental solutions of 
    $F(x,y)=m-a^2$.

We first define a map $F$ from $T_a$ to $S_a$ as follows. Let $(a, b, c)\in T_a$ and let $U, V$ be as given in Theorem \ref{MRS}. Then 
\begin{equation}\label{FS}
 F(a, b, c)=   
 \begin{cases}
 (c,b) & {\text{ if }} b\le V \\
 (-b,c-3ab) & {\text{ if }} b> V, c-3ab> 0 \\
 (b,0) & {\text{ if }} c-3ab=0.
\end{cases}   
\end{equation}

Observe that $F$ is well-defined, as 
if $b\le V$ then  from Theorem \ref{MRS} and Corollary \ref{Corollary3.1} we have that $(c, b)$ is a fundamental solution for $N=m-a^2$.

If $b>V$, then by Lemma \ref{either} we have 
$c-3ab\le V$. If $c-3ab=0$, then $F(b,0)=a$ with $b=\sqrt{m-a^2}$ and so by Theorem \ref{MRS}, part 2 we have $(b,0)$ is a fundamental solution.
Assume now that $0<c-3ab$. Note that $c-3ab=V$ is not possible as then $c-3ab=b$ by Lemma \ref{either} (and $b>V$). Thus 
$0<c-3ab<V$ and so by Theorem \ref{MRS} the solution $(-b, c-3ab)$ is  fundamental.

We proceed now to show that $F$ is surjective. Suppose that $(c,b)$ is a fundamental solution. 
If $b=0$, then $c>0$ (by Theorem \ref{MRS} part 2). Moreover, $m=a^2+c^2$ 
and $F(a, c, 3ac)=(c,0)$.

Next assume that $c>0$ (with $b>0$). Then 
$F(a,b, c)=(c,b)$ as $b\le V$. 

In the case when $c<0,$ note that
$(a, -c, b-3ac)$ is in $T_a$ by definition of minimality. Also from Theorem \ref{MRS} part 3 we see that $b\ne V$ (as $b=V$ implies that $u=c$ but $u>0$) and therefore $b<V$.  It follows that $-c>V$. Indeed assume that $c^2\le V$. Then as $\VV_3(a, -c, b-3ac)=(a, -c,-b)$ we have $F(-c, -b)=m-a^2$ and hence
$m-a^2= b^2+c^2+3ab|c|$ and thus
$$m-a^2<2\frac{m-a^2}{3a+2}+3a\frac{m-a^2}{3a+2}$$ which yields
$m-a^2<m-a^2$, a contradiction. Thus we have $-c>V$ and hence
$F(a, -c, b-3ac)=(c,b)$.

Next, we show that $F$ is injective. Suppose that  
$F(a, b, c)=F(a, b', c')$. If either $c=3ab$ or $c'=3a'b'$, it follows from (\ref{FS}) that  $c-3ab=c'-3a'b'=0$. Hence $b=b'$ which implies that $c=c'$. 

If both $b$ and $b'$ are less than or equal to $V$, then we have 
$F(a, b, c)= (c,b)=F(a, b', c')=(c',b')$ and it follows that the two triples are the same.  The case when both $b$ and $b'$ are greater than $V$ is analogous. Now we assume that $b\le V$ and $b'>V$. Clearly by definition of $F$, the images here cannot be equal as they are $(c,b)$ and 
$(-b', c'-3a'b')$ where $c$ and $b'$ are both positive.

Thus we have shown a bijection between $S_a$ and $T_a$. Observe that each minimal triple $(a, b, c)$  gives rise to two distinct fundamental solutions (one for $m-a^2$ and another for $m-b^2$), except when it is improper ($a=b$), and hence the formula given in the theorem follows.
\hfill$\square$

\vspace{0.7cm}

\noindent{\bf Proof of Theorem 1.3}

\noindent Suppose that $(1,b,c)$ is an $m$-Markoff triple. It follows that $F(b, c)=m-1$. Let $\gcd(b, c)=g$. Then $F\left(\frac{b}{g}, \frac{c}{g}\right)=\frac{m-1}{g^2}$ is 
a primitive representation and hence by Lemma \ref{4n}
we have $5\equiv x^2\pmod{4\frac{m-1}{g^2}}$ and so for every prime $p|\frac{m-1}{g^2}$ we have $(\frac{p}{5})\ne -1$ and the result follows.  Conversely let $m-1=S^2C$, where $C$ is square-free and satisfies the condition given in the lemma. Then
$5\equiv x^2\pmod C$ for some integer $x$.
Note that $C$ is odd  as $(\frac{2}{5})= -1$ and hence  $5\equiv y^2\pmod {4C}$ for some odd integer $y$ (if $x$ is even we consider $C-x$). It follows from 
Lemma \ref{4n} again, that there is a primitive representation of $C$ by some form of discriminant $d$. As the class number of $\mathbb{Q}(\sqrt{5})=1$, we have $F(b,c)=C$ and hence $F(Sb, Sc)=S^2C=m-1$. Thus $ R=(1, Sb, Sc)$ is a solution triple. If $Sb$ and $SC$ are both positive or both negative then $(1, |Sb|, |Sc|)$ is an $m$-Markoff triple. If $b$ or $c$ is less than or equal to $0$, then $\VV_2(R)$ or $\VV_3(R)$ respectively, is an $m$-Markoff triple and the proof of part 1 of the theorem is complete.

For the second part of the lemma,   we assume that there exists an $m-$Markoff triple $(1, b,c)$.  By the remarks just above Theorem 1.3, we have that the number of minimal triples $(1, x, y)$ is equal to the cardinality of  $T_1$ and hence of $S_1$. Therefore it remains to find all fundamental representations of $m-1$. It is straight forward to see by Theorem \ref{MRS} that if $F(b,c)=m-1$ is a fundamental solution and $\gcd(b,c)=g$, then $F\left(\frac{b}{g}, \frac{c}{g}\right)=\frac{m-1}{g^2}$ is also a fundamental solution. Moreover, it is a primitive representation. Conversely,  if $F(b, c)=\frac{m-1}{g^2}$ is a primitive fundamental solution, then 
$F(gb, gc)=m-1$ is a fundamental solution. Therefore to find all fundamental solutions of $m-1$ we need to find all the primitive fundamental solutions of 
$\frac{m-1}{g^2}$ for all possible $g$. Let us assume that there is a primitive fundamental solution of 
$\frac{m-1}{g^2}$ for some $g$.
Given the conditions on $A$, no prime divisor of $A$ divides $\frac{m-1}{g^2}$ and hence $A|g$. Also the highest power of $5$ that can divide  $\frac{m-1}{g^2}$ is $1$ (Lemma \ref{wn}) and so $5^{\alpha}|g$. Thus $g=5^{\alpha}Ad$, where $d|B$.  We have now shown that to count all fundamental solutions of $m-1$, we need to count all the fundamental primitive solutions of $\frac{B^2C}{d^2}$ where $d$ varies over all the divisors of $B$.
The claim now follows on the application of Lemma \ref{wn} to each integer $\frac{B^2C}{d^2}$. 
\hfill$\square$

\section{Questions and conjectures}

We devote this section to questions and conjectures about the number of minimal triples. We start with a few definitions to make precise our statements.

Note that if $(a, b, c)$ is a solution of (\ref{Mm}), then $(-a, -b,c)$ is also a solution. Hence we define the following transformations $\Ss_i$ that each gives rise to solution triples.

\begin{definition}\label{sign}
Let $(a,b,c)$ be a solution triple for (\ref{Mm}). The {\it sign transformations } $\Ss_1,\Ss_2,\Ss_3$ are defined as follows:
\begin{align*}
\Ss_1(a,b,c)=(a,-b,-c)\\
\Ss_2(a,b,c)=(-a,b, -c)\\
\Ss_3(a,b,c)=(-a,-b,c).
\end{align*}
\end{definition}

If $(a, b, c)$ is a minimal triple, then by definition we have $\phi=c-3ab\ge 0$. Two of the neighbouring triples (of  $\VV_3(a, b, c)$) give rise to the  triples $\Ss_2\VV_1\VV_3(a, b, c)=(3b\phi+a, b, \phi)$ and 
$\Ss_1\VV_2\VV_3(a, b, c)=(a,3a\phi+b, b, \phi)$, which once ordered are also minimal. We define the order of a minimal triple $(a,b,c)$  as the number of distinct minimal triples of the three in question. To make this definition precise, let us write $o(a,b,c)$ for the triple that is obtained after ordering its components.
\begin{definition}\label{def:order}
    Let $(a, b, c)$ be a minimal triple with 
    $\phi=c-3ab.$ Then the order is defined as 
    $$
     ord(a, b, c)=
    \begin{cases}
    1 & {\text{ if }} \phi=0 \\    
   \#\{(a, b, c),o(a,3a\phi+b,\phi), o(3b\phi+a, b, \phi)\}  
   & {\text{ if }} \phi\ne 0
    \end{cases}
    $$
\end{definition}
It is straightforward to verify that when $\phi\ne 0$ we have 
$$ord(a,b,c)=\#\{a, b, \phi\}.$$
We denote by $\mathcal{O}_1(m), \mathcal{O}_2(m)$ and $\mathcal{O}_3(m)$ the set of minimal triples of orders 1,2 and 3, respectively. Let $\mathcal{O}(m)$  be the set of minimal triples. 

Clearly $$\mathcal{O}(m)=\mathcal{O}_1(m)\cup\mathcal{O}_2(m)\cup\mathcal{O}_3(m).$$
Furthermore, $\#\mathcal{O}_2(m)$ is multiple of 2 and $\#\mathcal{O}_3(m)$
    is multiple of 3. In Table \ref{tab:minimal_tiples}, we present the set of minimal triples for $m\le 50$, listed according to their orders. 
 
One of the questions we are interested in is whether there can be exactly one minimal triple, that is $\OO(m)=1$. Recall that this is the case for the usual Markoff equation. We prove the following necessary condition in this case.
\begin{proposition}  
 If $m>1$ is such that there is a unique minimal triple $(a,b,c)$, then either 
$c=3ab$,
or, $a=b$ and 
$c=3a^2+a.$
\end{proposition} 
\begin{proof}
As $(a,b,c)$ is minimal, we have $\phi=c-3ab\ge 0$. If $\phi=0$, then $c=3ab$. Consider  $\phi>0$.  Note that as $(a, b, c)$ is minimal, the triples
$(\phi, b, 3b\phi+a)$ and 
$(\phi, a, 3a\phi+b)$, once ordered, are also minimal $m$-triples (using Definition \ref{minimaltriple}).
As there is exactly one minimal triple, these triples must equal  $(a,b,c)$. Since the maximum elements are $3b\phi+a$ and 
$3a\phi+b$, it follows that 
$\phi=a=b$. Hence $c-3a^2=a$ and the claim follows.
\end{proof}
\begin{remark}

Observe that for a minimal triple 
$(a, b, c)$ we  have $\phi=c-3ab=0$ if and only if $m=a^2+b^2$ is a sum of two non-zero squares. Every such representation of $m$ as a sum of two non-zero squares gives rise to a minimal triple of order $1$, namely $(a, b, 3ab)$. In Table 3 we list the first few values of $m$ for which there is only one minimal triple, and $m$ is not a sum of two squares.
\end{remark}
In the following proposition, we present a sufficient condition for the existence of only triples of order $3$.


\begin{proposition}
    Let $m$ be a positive integer such that $9m-4$ is prime and $m$ is not a sum of two squares. Then $\#\mathcal{O}(m)$ is divisible by $3$.
    \end{proposition}
    \begin{proof}
        If $(a, a, c)$ is an $m-$Markoff triple then it is easy to verify that 
        \begin{equation}\label{9m-4}
            9m-4=(3c-2)(3c-9a^2+2).       \end{equation}
           If $(a, a, c)$ is ordered, then $c>1$, so that $3c-2>1$. Also, $3c+2-9a^2$ is positive and clearly not equal to $1$ (looking at it modulo $3$). Thus $9m-4$ is not prime if there exists any triple with the first two components equal. As a result, 
             if $9m-4$ is prime, there are no such triples. It follows from  Definition \ref{def:order} that for minimal triples with $\phi\ne 0$, the order is $3$.  Moreover, as $m$ is not a sum of two squares, $\phi\ne 0$ and hence $\mathcal{O}(m)=\mathcal{O}_3(m)$ and the result follows.
    \end{proof}

   We are ready now to pose some questions and conjectures based on the above two propositions.
   
\textbf{Conjecture 1 } There are infinitely many natural numbers $m$ with exactly one minimal triple. Moreover the number of such $m$ up to $x$ is $O\left(\frac{x}{\log x}\right)$ (see Figure \ref{figure3} (a) and Table \ref{tab:uniqueminimal_tiples1}).

\textbf{Conjecture 2} There are infinitely many primes $m=p$ congruent to $1$ mod $4$ with exactly one minimal triple
(see Figure \ref{figure3} (b)).

\textbf{Conjecture 3} There are infinitely many natural numbers $m$ such that $\#\mathcal{O}(m)$ is congruent to $0\pmod 3$ (see Figures \ref{figure4} and \ref{figure5}).

\textbf{Question 1} Are there infinitely many natural numbers $m$ with exactly one minimal $m$ triple of the kind $(1,b, c)$? (See Figure \ref{figure6} (a) and Table \ref{table4}). 

\textbf{Question 2} Are there infinitely many natural numbers $m$ whose minimal triples are all of the kind $(1,b, c)$? (See Figure \ref{figure6} (b) and Table \ref{table4}).

\section{Computations}

In this section, we present some graphs and computations that support the conjectures given in Section 5, and that we hope will lead to new ideas and observations.


\vspace{1cm}

\begin{figure}[H]
  \centering
  \begin{subfigure}[b]{0.47\textwidth}
    \centering  \includegraphics[width=\textwidth]{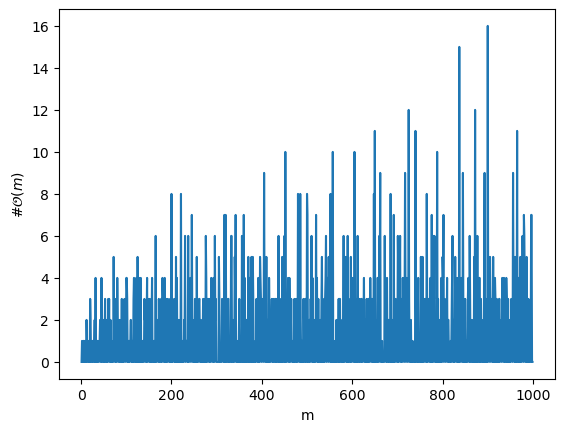}
    \caption{$m\leq 1000$}
    \label{fig:grafico1}
  \end{subfigure}
  \hfill
  \centering
  \begin{subfigure}[b]{0.47\textwidth}
    \centering  \includegraphics[width=\textwidth]{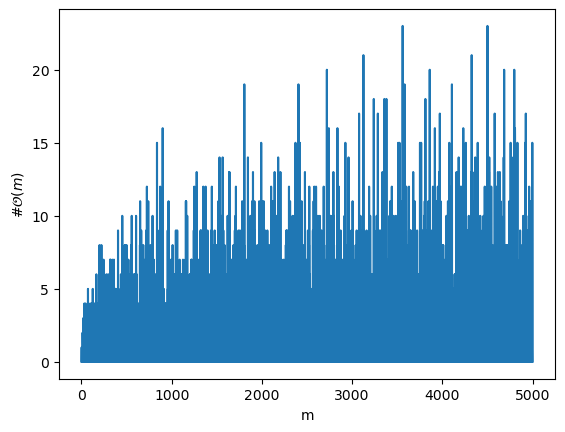}
    \caption{$m\leq 5000$}
    \label{fig:grafico7}
  \end{subfigure}
  \caption{Graph of the  function $\#\OO(m)$ }
  \label{figure2}
\end{figure}

\begin{figure}[H]
  \centering
  \begin{subfigure}[b]{0.47\textwidth}
    \centering  \includegraphics[width=\textwidth]{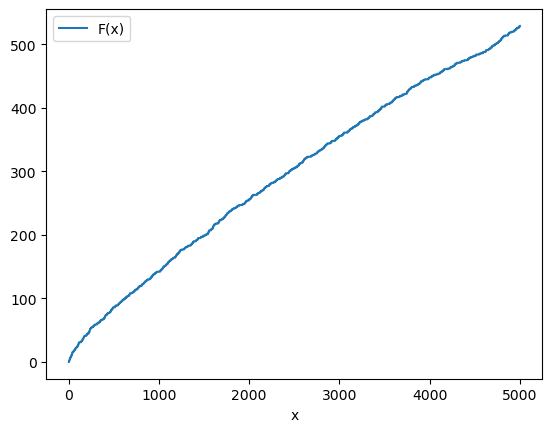}
    \caption{
    }
    \label{subfigure3}
  \end{subfigure}
  \hfill
  \centering
  \begin{subfigure}[b]{0.47\textwidth}
    \centering  \includegraphics[width=\textwidth]{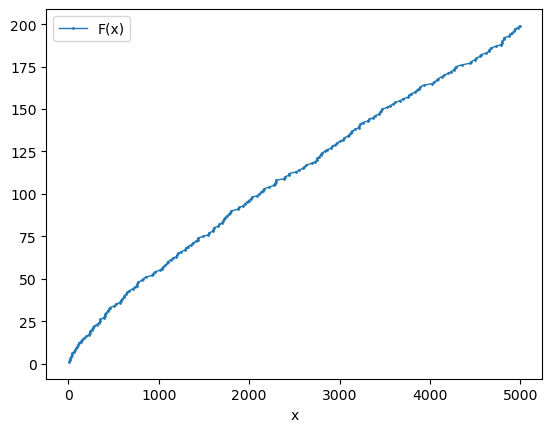}
    \caption{
  }
    \label{fig4}
  \end{subfigure}
  \caption{  $F(x)=$ number of $m\le x$ such that in a) $\#\OO(m)=1$ and in  (b) $\#\OO(m)=1$ with $m\equiv 1\pmod 4$ prime  }
  \label{figure3}
\end{figure}


\begin{figure}[H]
  \centering
  \begin{subfigure}[b]{0.45\textwidth}
    \centering  \includegraphics[width=\textwidth]{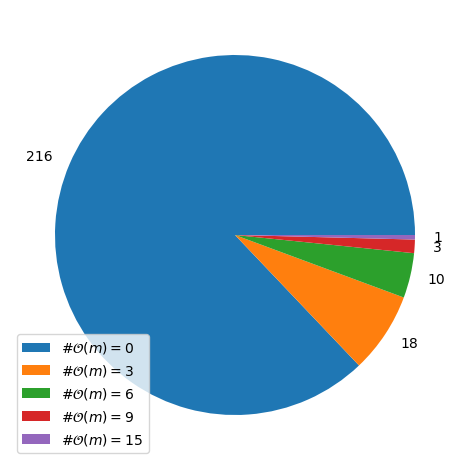}
    \caption{$m\leq 1000$}
    \label{fig:grafico3}
  \end{subfigure}
  \hfill
  \begin{subfigure}[b]{0.45\textwidth}
    \centering   \includegraphics[width=\textwidth]{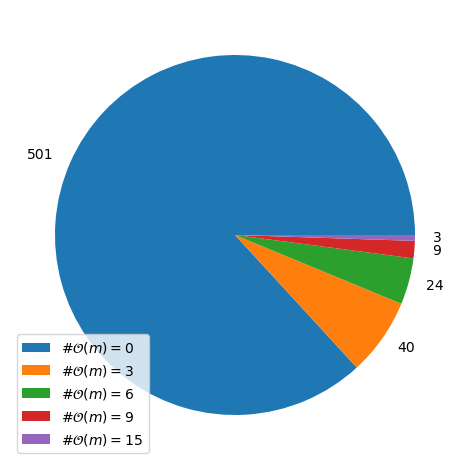}
    \caption{$m\leq 5000$}
    \label{fig:grafico8}
  \end{subfigure}
  \caption{Values of $\#\mathcal{O}(m)\equiv 0\pmod 3 $ with $9m - 4$  prime, $m\ne $ sum of two squares}

  \label{figure4}
\end{figure}

\begin{figure}[H]
  \centering
  \begin{subfigure}[b]{0.45\textwidth}
    \centering  \includegraphics[width=\textwidth]{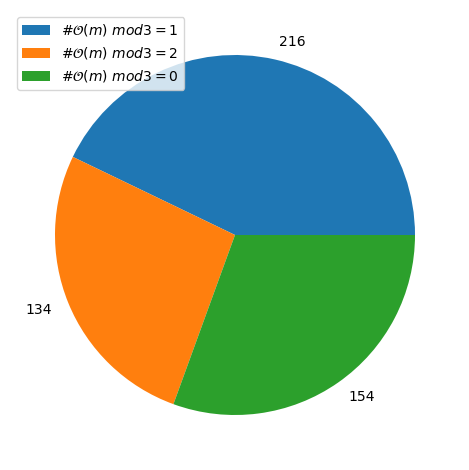}
    \caption{$m\leq 1000$}
    \label{fig:grafico4}
  \end{subfigure}
  \hfill
  \begin{subfigure}[b]{0.45\textwidth}
    \centering   \includegraphics[width=\textwidth]{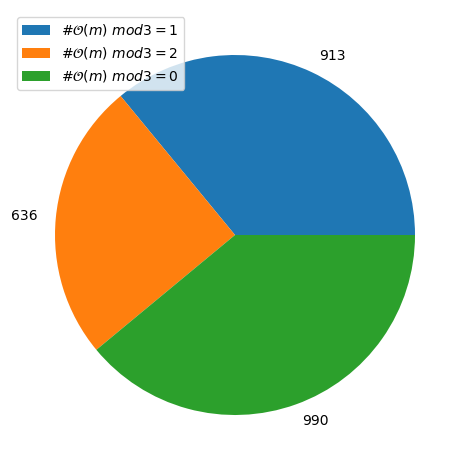}
    \caption{$m\leq 5000$}
    \label{fig:grafico2}
  \end{subfigure}
  \caption{Distribution of $\#\mathcal{O}(m)$  $mod\,\,3$, with $\#\mathcal{O}(m)\neq 0$  }
  \label{figure5}
\end{figure}




\begin{figure}[H]
  \centering
  \begin{subfigure}[b]{0.45\textwidth}
    \centering  \includegraphics[width=\textwidth]{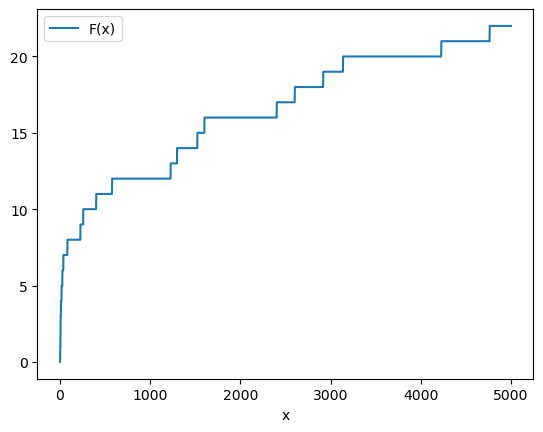}
    \caption{
  }
    \label{fig:grafico5}
  \end{subfigure}
  \hfill
  \begin{subfigure}[b]{0.45\textwidth}
    \centering   \includegraphics[width=\textwidth]{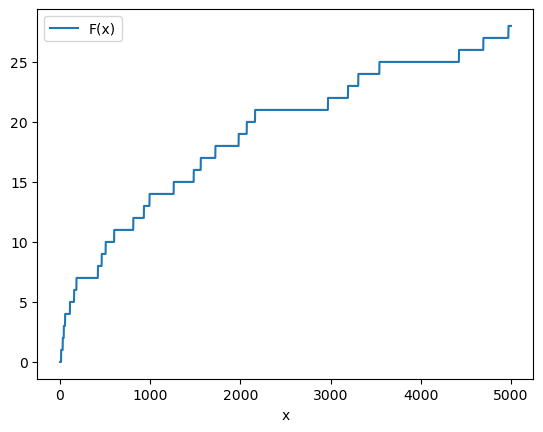}
    \caption{
  }
    \label{fig:grafico9}
  \end{subfigure}
  \caption{$F(x)=$ number of $m\le x$ such that  all minimal triples are of the form $(1,b,c)$, where in (a) $\#\OO(m)=1$ and in (b) $\#\OO(m)>1$ }
  \label{figure6}
\end{figure}

\begin{table}[H]
\begin{tabular}{|>{\centering\arraybackslash}p{0.5cm}|>{\centering\arraybackslash}p{2.2cm}|>{\centering\arraybackslash}p{4.1cm}|>{\centering\arraybackslash}p{6cm}|>{\centering\arraybackslash}p{1.5cm}|}
\hline
   $\bm{m}$ & $\bm{\mathcal{O}_1(m)}$& $\bm{\mathcal{O}_2(m)}$ & $\bm{\mathcal{O}_3(m)}$ & $\bm{\#\mathcal{O}(m)}$                               \\
\hline
   2 & \{(1, 1, 3)\}&&&1\\
\hline
   
   5 & \{(1, 2, 6)\}&&&1\\
   \hline
   6 & \{(1, 1, 4)\}&&&1\\
   \hline
   8 & \{(2, 2, 12)\} &&&1\\
     \hline
  10 & \{(1, 3, 9)\} &&&1\\
    \hline
  12 & & \{(1, 1, 5), (1, 2, 7)\}& &2\\
  \hline
  13 & \{(2, 3, 18)\}&&&1\\
  \hline
  17 & \{(1, 4, 12)\}&&&1\\
  \hline
  18 & \{(3, 3, 27)\}&&&1\\
  \hline
  20 & \{(2, 4, 24)\} & \{(1, 1, 6), (1, 3, 10)\} &  &3\\
    \hline
  21 & & \{(1, 2, 8), (2, 2, 13)\}& &2 \\
  \hline
  25 & \{(3, 4, 36)\}&&&1\\
  \hline
  26 & \{(1, 5, 15)\}&&&1\\
  \hline
  29 & \{(2, 5, 30)\}&&&1\\
  \hline
  30 & & \{(1, 1, 7), (1, 4, 13)\}& &2 \\
  \hline
  32 & \{(4, 4, 48)\} & & \{(1, 3, 11), (2, 3, 19),  (1, 2, 9)\}&4\\ 
  \hline
  34 & \{(3, 5, 45)\}&&&1\\
  \hline
  36 & \{(2, 2, 14)\}&&&1\\
  \hline
  37 & \{(1, 6, 18)\}&&&1\\
  \hline
  40 & \{(2, 6, 36)\}&&&1\\
  \hline
  41 & \{(4, 5, 60)\}&&&1\\
  \hline
  42 & & \{(1, 1, 8), (1, 5, 16)\}& &2\\
  \hline
  45 & \{(3, 6, 54)\}& & \{(1, 2, 10), (1, 4, 14), (2, 4, 25)\}&4\\
  \hline
  46 & & \{(1, 3, 12), (3, 3, 28)\}& &2\\
  \hline
 50 & \{(1, 7, 21), (5, 5, 75)\} & & & 2 \\
\hline
\end{tabular}
\caption{Set of minimal triples for $m\leq 50$ by order $1, 2, 3$}
\label{tab:minimal_tiples}
\end{table}

\begin{table}[H]
\begin{tabular}{|>{\centering\arraybackslash}p{0.7cm}|>{\centering\arraybackslash}p{6cm}|>{\centering\arraybackslash}p{6cm}|}
\hline
   $\bm{m}$ & $\bm{\mathcal{O}(m)\,\,with\,\, \phi=0}$& $\bm{\mathcal{O}(m)\,\,with\,\,\phi\neq 0}$\\
\hline
   2 & \{(1, 1, 3)\}&\\
\hline
   5 & \{(1, 1, 6)\}&\\
\hline
6 & & \{(1, 1, 4)\}\\
\hline
   8 & \{(2, 2, 12)\}&\\
\hline
   13 & \{(2,3,18)\}&\\
\hline
   17 & \{(1,4,12)\}&\\
\hline
   18 & \{(3,3,27)\}&\\
\hline
25 & \{(3,4,36)\}&\\
\hline
26 & \{(1,5,15)\}&\\
\hline
34 & \{(3,5,45)\}&\\
\hline
36 & & \{(2,2,14)\}\\
\hline
37 & \{(1,6,18)\}&\\
\hline
40 & \{(2,6,36)\}&\\
\hline
41 & \{(4,5,60)\}&\\
\hline
52 & \{(4,6,72)\}&\\
\hline
58 & \{(3,7,63)\}&\\
\hline
61 & \{(5,6,90)\}&\\
\hline
68 & \{(2,8,48)\}&\\
\hline
73 & \{(3,8,72)\}&\\
\hline
74 & \{(5,7,105)\}&\\
\hline
82 & \{(1,9,27)\}&\\
\hline
89 & \{(5,8,120)\}&\\
\hline
97 & \{(4,9,108)\}&\\
\hline
98 & \{(7,7,147)\}&\\
\hline
\end{tabular}
\caption{
 $m\leq 100$ with $\OO(m)=1$}
\label{tab:uniqueminimal_tiples1}
\end{table}

\begin{table}[H]
\centering
\begin{tabular}{|>{\centering\arraybackslash}p{1.5cm}|>{\centering\arraybackslash}p{6cm}|}
\hline
$\bm{m}$ & $\bm{\mathcal{O}(m)}$ \\
\hline
6 & \{(1, 1, 4)\}\\
\hline
36 & \{(2,2,14)\}\\
\hline
108 & \{(3,3,30)\}\\
\hline
1176 & \{(7,7,154)\}\\
\hline
61236 & \{(27,27,2214)\}\\
\hline
111078 & \{(33,33,3300)\}\\
\hline
156066 & \{(37,37,4144)\}\\
\hline
405756& \{(51,51,7854)\}\\
\hline
\end{tabular}
\caption{
 $m\leq 405756$ with  $\phi\neq 0$ and $\#\OO(m)=1$ 
}
\label{table3}
\end{table}

\begin{table}[H]
\centering
\begin{tabular}{|>{\centering\arraybackslash}p{1.5cm}|>{\centering\arraybackslash}p{6cm}|}
\hline
$\bm{m}$ & $\bm{\mathcal{O}(m)}$ \\
\hline
5 & \{(1, 2, 6)\}\\
\hline
6 & \{(1, 1, 4)\}\\
\hline
10 & \{(1, 3, 9)\}\\
\hline
12 & \{(1, 1, 5), (1,2,7)\}\\
\hline
17 & \{(1, 4, 12)\}\\
\hline
26 & \{(1,5,15)\}\\
\hline
37 & \{(1,6,18)\}\\
\hline
42 & \{(1,1,8), (1,5,16)\}\\
\hline
56 & \{(1,1,9), (1,6,19)\}\\
\hline
82 & \{(1,9,27)\}\\
\hline
110 & \{(1,1,12), (1,9,28)\}\\
\hline
156 & \{(1,1,14),(1,11,34)\}\\
\hline
182 & \{(1,1,15), (1,12,37)\}\\
\hline
226 & \{(1,15,42)\}\\
\hline
257 & \{(1,16,48)\}\\
\hline
401 & \{(1,20,60)\}\\
\hline
420& \{(1,1,22), (1,19,58)\}\\
\hline
462& \{(1,1,23), (1,20,61)\}\\
\hline
506& \{(1,1,24), (1,21,64)\}\\
\hline
577& \{(1,24,72)\}\\
\hline
600& \{(1,1,26), (1,23,70)\}\\
\hline
812& \{(1,1,30),(1,27,82)\}\\
\hline
930& \{(1,1,32), (1,29,88)\}\\
\hline
992& \{(1,1,33), (1,30,91)\}\\
\hline
\end{tabular}
\caption{
 $m\leq 1000$ such that all minimal triples have first component equal to $1$
}
\label{table4}
\end{table}


\begin{thebibliography}{999}










\bibitem{TWX} Colliot-Th\'el\`{e}ne, J. L., Wei, D. and  Xu, F., {\it Brauer-Manin obstruction for Markoff surfaces}, Ann. Sc. Norm. Super. Pisa Cl. Sci., (5), Vol. XXI (2020), 1257-1313.


\bibitem{GS} Ghosh, A. and  Sarnak, P. {\it Integral points on Markoff type cubic surfaces}, Invent. math. 229,  (2022), 689-749.  


\bibitem{HW} G. H. Hardy and E. M. Wright,
 {\it An introduction to 
 the theory of numbers}, { Oxford science publications, $5^{\rm th}$ e.},
 1979.

 
\bibitem{LM}Loughran, D. and Mitakin, V., {\it Integral Hasse principle and strong approximation for Markoff surfaces}, International Mathematics Research Notices, Volume 2021, Issue 18, (2021), 14086-14122,

\bibitem{LS}
 Luca, F. and Srinivasan, A., {\it Markov equation with Fibonacci components}, The Fibonacci Quarterly, {\bf 56}, no. 2 (2018), 126-129.
\bibitem{M1} 
Markoff, A.A., {\it Sur les formes quadratiques binaires ind\'efinies},
Mathe\-matische Annalen \textbf{15}, (1879), 381-496.

\bibitem{M2}
Markoff, A.A., {\it Sur les formes quadratiques binaires ind\'efinies
(second m\'emoire)}, Mathematische Annalen \textbf{17}, (1880), 379-399.


\bibitem{MRS}{Matthews, K., Robertson, J. and Srinivasan, A.}, 
{\it On fundamental solutions of binary quadratic form equations,}
Acta Arithmetica, {\textbf{169} (3)} (2015), 291-299.


\bibitem{Mor}
{Mordell L. J., On the Integer Solutions of the Equation $x^2+y^2+z^2+2xyz = n$, Journal of the London Mathematical Society, Volume s1-28, Issue 4, (1953), 500-510. 





\bibitem{Ri} Ribenboim, P., \textit{My Numbers, My Friends, Popular
Lectures on Number Theory}, { Springer-Verlag}, 2000.

\bibitem{Sr} Srinivasan, A., {\it D(-1)-quadruples and products of two primes}, Glasnik Matematicki, {\bf 50}, no. 2, (2015), 261-268. 


}\end{thebibliography}
\end{document}